\documentclass[11pt,leqno]{amsart}
 \usepackage{graphicx}    
 \usepackage{amsmath}
\usepackage{amssymb}
\usepackage[T1]{fontenc}
\usepackage[latin1]{inputenc}
\usepackage[english]{babel}
\usepackage[arrow, matrix, curve]{xy}
\usepackage{amsthm}
\usepackage{amsmath,amscd}
\usepackage{mathrsfs}
\usepackage{epic,eepic}
\numberwithin{equation}{section}

\DeclareMathOperator\GL{GL}
\DeclareMathOperator\id{id}

\DeclareMathOperator\coker{coker}

\DeclareMathOperator\Spf{Spf}

\DeclareMathOperator\Spa{Spa}

\DeclareMathOperator\pr{pr}
\DeclareMathOperator\Fil{Fil}

\DeclareMathOperator\rig{rig}
\DeclareMathOperator\ad{ad}

\DeclareMathOperator\Ad{Ad}
\DeclareMathOperator\lft{lft}

\renewcommand{\phi}{\varphi}

\newcommand{\Rscr}{\mathscr{R}}

\newcommand{\Acal}{\mathscr{A}}
\newcommand{\Bcal}{\mathscr{B}}

\newcommand{\Ecal}{\mathcal{E}}
\newcommand{\Fcal}{\mathcal{F}}

\newcommand{\Mcal}{\mathcal{M}}
\newcommand{\Ncal}{\mathcal{N}}
\newcommand{\Ocal}{\mathcal{O}}

\newcommand{\Ucal}{\mathcal{U}}
\newcommand{\Xcal}{\mathcal{X}}
\newcommand{\Vcal}{\mathcal{V}}

\newcommand{\Q}{\mathbb{Q}}

\newcommand{\Z}{\mathbb{Z}}
\newcommand{\C}{\mathbb{C}}

\newcommand{\Abb}{\mathbb{A}}
\newcommand{\boldB}{\mathbb{B}}
\newcommand{\Fbb}{\mathbb{F}}

\newcommand{\Pbb}{\mathbb{P}}
\newcommand{\Ubb}{\mathbb{U}}

\newcommand{\bfA}{{\bf A}}
\newcommand{\bfB}{{\bf B}}
\newcommand{\bfE}{{\bf E}}

\newcommand{\Dfrak}{\mathfrak{D}}
\newcommand{\Mfrak}{\mathfrak{M}}
\newcommand{\Nfrak}{\mathfrak{N}}

\newcommand{\mfrak}{\mathfrak{m}}

\newtheorem{theo}{Theorem}[section]
\newtheorem{lem}[theo]{Lemma}
\newtheorem{prop}[theo]{Proposition}
\newtheorem{cor}[theo]{Corollary}
\newtheorem{conj}[theo]{Conjecture}
\theoremstyle{remark}
\newtheorem{rem}[theo]{Remark}
\theoremstyle{remark}
\newtheorem{expl}[theo]{Example}
\theoremstyle{definition}
\newtheorem{defn}[theo]{Definition}

\begin{document}

\title[Families of $p$-adic Galois representations ]{Families of  $p$-adic Galois representations and $(\phi,\Gamma)$-modules}
\author[E. Hellmann]{Eugen Hellmann}
\begin{abstract}
We investigate the relation between $p$-adic Galois representations and overconvergent $(\phi,\Gamma)$-modules in families.
Especially we construct a natural open subspace of a family of $(\phi,\Gamma)$-modules, over which it is induced by a family of Galois-representations.
\end{abstract}
\maketitle

\section{Introduction}
In \cite{BergerColmez}, Berger and Colmez generalized the theory of overconvergent $(\phi,\Gamma)$-modules to families parametrized by $p$-adic Banach algebras. 
More preceisely their result gives a fully faithful functor from the category of vector bundles with continuous Galois action on a rigid analytic variety to the category of families of \'etale overconvergent $(\phi,\Gamma)$-modules.
This functor fails to be essentially surjective. However it was shown by Kedlaya and Liu in \cite{KedlayaLiu} that this functor can be inverted locally around rigid analytic points. 

It was already pointed out in our previous paper \cite{families} that the right category to handle these objects is the category of adic spaces (locally of finite type over $\Q_p$) as introduced by Huber, see \cite{Huber}.
Using the language of adic spaces, we show in this paper that given a family $\Ncal$ of $(\phi,\Gamma)$-modules over the relative Robba ring $\Bcal_{X,\rig}^\dagger$ on an adic space $X$ locally of finite type over $\Q_p$ (see below for the construction of the sheaf $\Bcal_{X,\rig}^\dagger$), one can construct natural open subspaces $X^{\rm int}$ resp.~$X^{\rm adm}$, where the family $\Ncal$ is \'etale resp.~induced by a family of Galois representations. 

Our main results are as follows. Let $K$ be a finite extension of $\Q_p$ and write $G_K$ for its absolute Galois group. Further we fix a cyclotomic extension $K_\infty=\bigcup K(\mu_{p^n})$ of $K$ and write $\Gamma={\rm Gal}(K_{\infty}/K)$. 
\begin{theo}
Let $X$ be a reduced adic space locally of finite type over $\Q_p$, and let $\Ncal$ be a family of $(\phi,\Gamma)$-modules over the relative Robba ring $\Bcal_{X,\rig}^\dagger$.\\
\noindent {\rm (i)} There is a natural open subspace $X^{\rm int}\subset X$ such that the restriction of $\Ncal$ to $X^{\rm int}$ is \'etale, i.e. locally on $X^{\rm int}$ there is a family of \'etale lattices $\Nfrak\subset \Ncal$.\\
\noindent {\rm (ii)} The formation $(X,\Ncal)\mapsto X^{\rm int}$ is compatible with base change in $X$, and $X=X^{\rm int}$ whenever the family $\Ncal$ is \'etale.
\end{theo}
In the classical theory of overconvergent $(\phi,\Gamma)$-modules, the slope filtration theorem of Kedlaya, \cite[Theorem 1.7.1]{Kedlaya} asserts that a $\phi$-module over the Robba ring admits an \'etale lattice if and only if it is purely of slope zero. The latter condition is a semi-stability condition which only involves the slopes of the Frobenius. The question whether there is a generalization of this result to $p$-adic families was first considered by R. Liu in \cite{Liu}, where he shows that an \'etale lattice exists locally around rigid analytic points. Here we show that Kedlaya's result does not generalize to families. That is, we construct a family of $\phi$-modules which is \'etale at all rigid analytic points but which is not \'etale as a family of $(\phi,\Gamma)$-modules (in the sense specified below).

On the other hand, we construct an \emph{admissible} subset $X^{\rm adm}\subset X$ for a family of $(\phi,\Gamma)$-modules over $X$. This is the subset over which there exists a family of Galois representations. It will be obvious that we always have an inclusion $X^{\rm adm}\subset X^{\rm int}$.  
\begin{theo}
Let $X$ be a reduced adic space locally of finite type over $\Q_p$ and $\Ncal$ be a family of $(\phi,\Gamma)$-modules over the relative Robba ring $\Bcal_{X,\rig}^\dagger$.\\
\noindent {\rm (i)} There is a natural open and partially proper subspace $X^{\rm adm}\subset X$ and a family $\Vcal$ of $G_K$-representations on $X^{\rm adm}$ such that $\Ncal|_{X^{\rm adm}}$ is associated to $\Vcal$ by the construction of Berger-Colmez. \\
\noindent {\rm (ii)} The formation $(X,\Ncal)\mapsto (X^{\rm adm},\Vcal)$ is compatible with base change in $X$, and $X=X^{\rm adm}$ whenever the family $\Ncal$ comes from a family of Galois representations. \\
\noindent {\rm (iii)} Let $X$ be of finite type and let $\Xcal$ be a formal model of $X$. Let $Y\subset X$ be the tube of a closed point in the special fiber of $\Xcal$. If $Y\subset X^{\rm int}$, then $Y\subset X^{\rm adm}$.
\end{theo}

In a forthcoming paper we will apply the theory developed in this article to families of trianguline $(\phi,\Gamma)$-modules and give an alternative construction of Kisin's finite slope space. 

{\bf Acknowledgements:} I thank R. Liu, T. Richarz, P. Scholze and M. Rapoport for helpful conversations. The author was partially supported by the SFB TR 45 of the DFG (German Research Foundation). 

\section{Sheaves of period rings}
In this section we define relative versions of the classical period rings used in the theory of $(\phi,\Gamma)$-modules and in $p$-adic Hodge-theory.
Some of these sheaves were already defined in \cite[8]{families}.

Let $K$ be a finite extension of $\Q_p$ with ring of integers $\Ocal_K$ and residue field $k$. Fix an algebraic closure $\bar K$ of $K$ and write $G_K={\rm Gal}(\bar K/K)$ for the absolute Galois group of $K$. As usual we choose a compatible system $\epsilon_n\in\bar K$ of $p^n$-th root of unity and write $K_{\infty}=\bigcup K(\epsilon_n)$. Let $H_K\subset G_K$ denote the absolute Galois group of $K_{\infty}$ and write $\Gamma={\rm Gal}(K_{\infty}/K)$. Finally we denote by $W=W(k)$ the ring of Witt vectors with coefficients in $k$ and by $K_0={\rm Frac}\, W$ the maximal unramified extension of $\Q_p$ inside $K$.
\subsection{The classical period rings}
We briefly recall the definitions of the period rings, as defined in \cite{Berger1} for example. 
Write 
\[\tilde \bfE^+=\lim\limits_{\substack{\longleftarrow \\ x\mapsto x^p}} \Ocal_{\C_p}/p\Ocal_{\C_p}.\]
This is a perfect ring of characteristic $p$ which is complete for the valuation ${\rm val}_{\bfE}$ given by ${\rm val}_{\bfE}(x_0,x_1,\dots)={\rm val}_p(x_0)$.
Let 
\[\tilde\bfE={\rm Frac}\,\tilde\bfE^+=\tilde\bfE^+[\tfrac{1}{\underline{\epsilon}}],\]
where $\underline\epsilon=(\epsilon_1,\epsilon_2,\dots)\in\tilde\bfE^+$. Further we define 
\begin{align*}
\tilde\bfA^+&=W(\tilde\bfE^+) & \tilde\bfA&=W(\tilde\bfE),\\
\tilde\bfB^+&=\tilde\bfA^+[1/p] & \tilde\bfB&=\tilde\bfA[1/p].
\end{align*}
On all these ring we have an action of the Frobenius morphism $\phi$ which is induced by the $p$-th power map on $\tilde\bfE$.
Further we consider the ring $\bfA_K$ which is the $p$-adic completion of $W((T))$ and denote by $\bfB_K=\bfA_K[1/p]$ its rational analogue. We embed these rings into $\tilde\bfB$ by mapping $T$ to $[\underline{\epsilon}]-1$. The morphism $\phi$ induces the endomorphism $T\mapsto (T+1)^p-1$ on $\bfA_K$, resp. $\bfB_K$. Further $G_K$ acts on $\bfA_K$ through the quotient $G_K\rightarrow \Gamma$.

For $r<s\in\Z$ we define
\begin{align*}
\bfA^{[r,s]}&=\left \{\sum_{n\in\Z}a_nT^n\left | a_n\in K_0,{\begin{array}{*{20}c}0\leq {\rm val}_p(a_{n}p^{n/r})\rightarrow \infty,\, n\rightarrow -\infty\\ 0\leq {\rm val}_p(a_np^{n/s})\rightarrow \infty,\, n\rightarrow \infty\end{array}}\right\},\right.\\
\tilde\bfA^{\dagger,r}&=\left \{\sum_{n\in\Z}[x_n]p^n\left| x_n\in\tilde\bfE, 0\leq {\rm val}_{\bfE}(x_n)+\tfrac{prn}{p-1}\rightarrow \infty,\, n\rightarrow \infty\right\},\right.\\
\tilde\bfB^{\dagger,r}&=\left \{\sum_{n\in\Z}[x_n]p^n\left| x_n\in\tilde\bfE,  {\rm val}_{\bfE}(x_n)+\tfrac{prn}{p-1}\rightarrow \infty,\, n\rightarrow \infty\right\}.\right.
\end{align*}
The rings $\tilde\bfA^{\dagger,r}$ and $\tilde\bfB^{\dagger, r}$ are endowed with the valuation 
\[w_r:\sum p^k[x_k]\longmapsto\inf_k\big\{{\rm val}_{\bfE}(x_k)+\tfrac{prk}{p-1}\big\}.\]
Using these definitions the usual period rings are defined as follows:
\begin{equation}\label{tilderings}
\begin{aligned}
\tilde\bfB_{\rig}^{\dagger,s}&=\text{Frechet completion of}\ \tilde\bfB^{\dagger, s}\ \text{for the valuations}\ w_{s'},\,s'\geq s,\\
\tilde\bfB^{\dagger}&=\lim\nolimits_\rightarrow \tilde\bfB^{\dagger,s},\\ \tilde\bfB^{\dagger}_{\rig}&=\lim\nolimits_\rightarrow \tilde\bfB^{\dagger,s}_{\rig}.
\end{aligned}
\end{equation}
\begin{equation}\label{periodrings}
\begin{aligned}
\bfB^{[r,s]}&=\bfA^{[r,s]}[1/p],  & \bfA^{\dagger,r}&=\lim\nolimits_\leftarrow \bfA^{[r,s]},\\
\bfB^{\dagger,s}&=\bfB_K\cap \tilde\bfB^{\dagger,s},  & \bfB^{\dagger,r}_{\rig}&=\lim\nolimits_\leftarrow \bfB^{[r,s]},\\ 
\bfA^{\dagger}&=\bfA_K\cap\bfB^\dagger, & \bfB^{\dagger}&=\lim\nolimits_\rightarrow\bfB^{\dagger,s},\\ 
\bfB^{\dagger}_{\rig}&=\lim\nolimits_\rightarrow \bfB^{\dagger,s}_{\rig}.  
\end{aligned}
\end{equation}

Note that these definitions equip all rings with a canonical topology. 
There is a canonical action of $G_K$ on all of these rings. This action is continuous for the canonical topology. The $H_K$-invariants of $\tilde{\bf R}$ for any of the rings in $(\ref{tilderings})$ are given by the corresponding ring without a tilde ${\bf R}$ in $(\ref{periodrings})$, where ${\bf R}$ is identified with a subring of $\tilde{\bf R}$ by $T\mapsto [\underline{\epsilon}]-1$. Hence there is a natural continuous $\Gamma$-action on all the rings in $(\ref{periodrings})$.
The Frobenius endomorhpism $\phi$ of $\tilde\bfB$ induces a ring homomorphisms
\begin{align*}
\bfA^{[r,s]}&\longrightarrow \bfA^{[pr,ps]}\\
\tilde\bfA^{\dagger,s }&\longrightarrow \tilde\bfA^{\dagger,ps},
\end{align*}
for $r,s\gg 0$ and in the limit endomorphisms of the rings
\[\bfA^\dagger,\bfB^\dagger,\bfB^\dagger_{\rig},\tilde \bfB^\dagger,\tilde\bfB^\dagger_{\rig}.\]
These homomorphisms will be denoted by $\phi$ and commute with the action of $\Gamma$, resp.~$G_K$.

\begin{rem}\label{geomabsolute} Let us points out that some of the above rings have a geometric interpretation. We write $\boldB$ for the closed unit disc over $K_0$ and $\Ubb\subset \boldB$ for the open unit disc. Then 
\begin{align*}
\bfA^{[r,s]}&=\Gamma(\boldB_{[p^{-1/r},p^{-1/s}]},\Ocal^+_{\boldB}), & \bfB^{[r,s]} &=\Gamma(\boldB_{[p^{-1/r},p^{-1/s}]},\Ocal_{\boldB}),\\
\bfA^{\dagger,s}&=\Gamma(\Ubb_{\geq p^{-1/s}},\Ocal^+_{\Ubb}), & \bfB^{\dagger,s}_{\rm rig }&=\Gamma(\Ubb_{\geq p^{-1/s}},\Ocal_{\Ubb}),
\end{align*}
where $\boldB_{[a,b]}\subset \boldB$ is the subspace of inner radius $a$ and outer radius $b$ and $\Ubb_{\geq a}\subset \Ubb$ is the subspace of inner radius $a$. 
\end{rem}

\subsection {Sheafification}
Let $X$ be an adic space locally of finite type over $\Q_p$. Recall that $X$ comes along with a sheaf $\Ocal^+_X\subset \Ocal_X$ of open and integrally closed subrings.

Let $A^+$ be a reduced $\Z_p$-algebra topologically of finite type. Recall that for $i\geq 0$ the completed tensor products
\[A^+\widehat{\otimes}_{\Z_p}W_i(\tilde\bfE^+)\ \text{and}\ A^+\widehat{\otimes}_{\Z_p}W_i(\tilde\bfE)\]
are the completions of the ordinary tensor product for the topology that is given by the discrete topology on $A^+/p^iA^+$ and by the natural topology on $W_i(\tilde\bfE^+)$ resp.~$W_i(\tilde{\bfE})$, see \cite[8.1]{families}.

Let $X$ be a \emph{reduced} adic space locally of finite type over $\Q_p$. As in \cite[8.1]{families} we can define sheaves $\tilde{\mathscr{E}}_X^+$, $\tilde{\mathscr{E}}_X$,  $\tilde\Acal_X^+$ and $\tilde\Acal_X$ by demanding
\begin{align*}
\Gamma(\Spa(A,A^+),\tilde{\mathscr{E}}^+_X)&=A^+\widehat{\otimes}_{\Z_p}\tilde{\bfE}^+,\\
\Gamma(\Spa(A,A^+),\tilde{\mathscr{E}}_X)&=A^+\widehat{\otimes}_{\Z_p}\tilde{\bfE},\\
\Gamma(\Spa(A,A^+),\tilde\Acal_X^+)&=\lim\limits_{\longleftarrow\ i} A^+\widehat{\otimes}_{\Z_p}W_i(\tilde{\bfE}^+),\\
\Gamma(\Spa(A,A^+),\tilde\Acal_X)&=\lim\limits_{\longleftarrow\ i} A^+\widehat{\otimes}_{\Z_p}W_i(\tilde{\bfE}),
\end{align*}
for an affinoid open subset $\Spa(A,A^+)\subset X$.

We define the sheaf $\Acal_{X,K}$ to be the $p$-adic completion of $(\Ocal_X^+\otimes_{\Z_p}W)((T))$. Further we set $\Bcal_{X,K}=\Acal_{X,K}[1/p]$.

Let $A^+$ be as above and $A=A^+[1/p]$. We define 
\[A^+\widehat{\otimes}_{\Z_p}\bfA^{[r,s]}\ \text{and}\ A^+\widehat{\otimes}_{\Z_p}\tilde\bfA^{\dagger,s}\]
to be the completion of the ordinary tensor product for the $p$-adic topology on $A^+$ and the natural topology on $\bfA^{[r,s]}$ resp.~$\tilde\bfA^{\dagger,s}$. These completed tensor products can be viewed as subrings of $\Gamma(\Spa(A,A^+),\tilde\Acal_{\Spa(A,A^+)})$. For a reduced adic space $X$ locally of finite type over $\Q_p$, we define the sheaves $\Acal_X^{[r,s]}$ and $\tilde\Acal_X^{\dagger,s}$ by demanding
\begin{align*}
\Gamma(\Spa(A,A^+),\Acal_X^{[r,s]})&=A^+\widehat{\otimes}_{\Z_p}\bfA^{[r,s]},\\
\Gamma(\Spa(A,A^+),\tilde\Acal_X^{\dagger, s})&=A^+\widehat{\otimes}_{\Z_p}\tilde\bfA^{\dagger,s},
\end{align*}
for an open affinoid $\Spa(A,A^+)\subset X$. Similarly we define the sheaf $\tilde\Bcal_X^{\dagger,s}$.
Finally, as in the case above, we can use these sheaves to define the sheafified versions of $(\ref{periodrings})$ and $(\ref{tilderings})$.
\begin{equation}\label{tildersheaves}
\begin{aligned}
\tilde\Bcal_{X,\rig}^{\dagger,s}&=\text{Frechet completion of}\ \tilde\Bcal_X^{\dagger, s}\ \text{for the valuations}\ w_{s'},\,s'\geq s,\\
\tilde\Bcal_X^{\dagger}&=\lim\nolimits_\rightarrow \tilde\Bcal_X^{\dagger,s},\\ \tilde\Bcal^{\dagger}_{X,\rig}&=\lim\nolimits_\rightarrow \tilde\Bcal^{\dagger,s}_{X,\rig}.
\end{aligned}
\end{equation}
\begin{equation}\label{periosheaves}
\begin{aligned}
\Bcal_X^{[r,s]}&=\Acal_X^{[r,s]}[1/p],  & \Acal_X^{\dagger,r}&=\lim\nolimits_\leftarrow \Acal_X^{[r,s]},\\
\Bcal_X^{\dagger,s}&=\Bcal_{X,K}\cap \tilde\Bcal_X^{\dagger,s},  & \Bcal^{\dagger,r}_{X,\rig}&=\lim\nolimits_\leftarrow \Bcal_X^{[r,s]},\\ 
\Acal_X^{\dagger}&=\Acal_{X,K}\cap\Bcal_X^\dagger, & \Bcal_X^{\dagger}&=\lim\nolimits_\rightarrow\Bcal_X^{\dagger,s},\\ 
\Bcal^{\dagger}_{X,\rig}&=\lim\nolimits_\rightarrow \Bcal^{\dagger,s}_{X,\rig}.  
\end{aligned}
\end{equation}

Note that all the rational period rings (i.e. those period rings in which $p$ is inverted) can also be defined on a non-reduced space $X$ by locally embedding the space into a reduced space $Y$ and restricting the corresponding period sheaf from $Y$ to $X$,
compare \cite[8.1]{families}.
\begin{rem}
As in the absolute case there is a geometric interpretation of some of these sheaves of period rings. 
\begin{align*}
\Acal_X^{[r,s]}&=\pr_{X,\ast}\big(\Ocal^+_{X\times \boldB_{[p^{-1/r},p^{-1/s}]}}\big), & \Bcal_X^{[r,s]} &=\pr_{X,\ast}\big(\Ocal_{X\times \boldB_{[p^{-1/r},p^{-1/s}]}}\big),\\
\Acal_X^{\dagger,s}&=\pr_{X,\ast}\big(\Ocal^+_{X\times \Ubb_{\geq p^{-1/s}}}\big), & \Bcal^{\dagger,s}_{X,\rm rig }&=\pr_{X,\ast}\big(\Ocal_{X\times \Ubb_{\geq p^{-1/s}}}\big).
\end{align*}
Here $\pr_X$ denotes the projection from the product to $X$. Especially we can define the sheaves $\Acal_X^{[r,s]}$, $\Acal_X^{\dagger,s}$ and $\Acal_X^\dagger$ on every adic spaces, not only on reduced spaces.
\end{rem}

By construction all the sheaves $\tilde{\Rscr}_X$ (i.e. those of the period sheaves with a tilde) are endowed with a continuous $\Ocal_X$-linear $G_K$-action and an endomorphism $\phi$ commuting with the Galois action.
The sheaves $\Rscr_X$ (i.e. those period rings without a tilde) are endowed with a continuous $\Gamma$-action and an endomorphism $\phi$ commuting with the action of $\Gamma$. 

In the following we will use the notation $X(\bar\Q_p)$ for the set of rigid analytic points of an adic space $X$ locally of finite type over $\Q_p$, i.e. $X(\bar\Q_p)=\{x\in X\mid k(x)/\Q_p\ \text{finite}\}$.
\begin{prop} Let $X$ be a reduced adic space locally of finite type over $\Q_p$ and let ${\bf R}$ be any of the integral period rings {\rm (}i.e. a period ring in which $p$ is not inverted{\rm )} defined above. Let $\Rscr_X$ be the corresponding sheaf of period rings on $X$.\\
\noindent {\rm (i)} The canonical map
\[\Gamma(X,\Rscr_X)\longrightarrow \prod_{x\in X(\bar\Q_p)} k(x)^+\otimes_{\Z_p}{\bf R}\]
is an injection.\\
\noindent {\rm (ii)} Let ${\bf R}'\subset {\bf R}$ be another integral period ring with corresponding sheaf of period rings $\Rscr'_X\subset\Rscr_X$ and let $f\in \Rscr_X$. Then $f\in\Rscr'_X$ if and only if 
\[f(x)\in k(x)^+\otimes_{\Z_p}{\bf R}'\subset k(x)^+\otimes{\bf R}\]
for all rigid analytic points $x\in X$.
\end{prop}
\begin{proof}
This is proven along the same lines as \cite[Lemma 8.2]{families} and \cite[Lemma 8.6]{families}.
\end{proof}

\begin{cor}
Let $X$ be an adic space locally of finite type over $\Q_p$, then
\begin{align*}
\big(\tilde\Bcal_{X,\rig}^\dagger\big)^{\phi=\id}&=\Ocal_X & \big(\tilde\Bcal_{X,\rig}^\dagger\big)^{H_K}&=\Bcal^\dagger_{X,\rig},\\
\big(\tilde\Bcal_{X}^\dagger\big)^{\phi=\id}&=\Ocal_X & \big(\tilde\Bcal_{X}^\dagger\big)^{H_K}&=\Bcal^\dagger_{X}.
\end{align*}
\end{cor}
\begin{proof}
If the space is reduced this follows from the above by chasing through the definitions. Otherwise we can locally on $X$ choose a finite morphism to a reduced space $Y$ (namely a polydisc) and study the $\phi$- resp.~$H_K$-invariants in the fibers over the rigid analytic points of $Y$, compare \cite[Corollary 8.4, Corollary 8.8]{families}
\end{proof}
\begin{rem}
Let $X$ be an adic space locally of finite type and $\mathscr{R}$ be any of the sheaves of topological rings defined above. If $x\in X$ is a point then we will sometimes write $\mathscr{R}_x$ for the completion of the fiber $\mathscr{R}\otimes k(x)$ of $\mathscr{R}$ at $x$ with respect to the canonical induced topology. 
\end{rem}
\section{Coherent $\Ocal_X^+$-modules and lattices}

As the notion of being \'etale is defined by using lattices we make precise what we mean by (families of) lattices.

Let $X$ be an adic space locally of finite type over $\Q_p$. The space $X$ is endowed with a structure sheaf $\Ocal_X$ and a sheaf of open and integrally closed subrings $\Ocal_X^+\subset \Ocal_X$ consisting of the power bounded sections of $\Ocal_X$. Recall that for any ringed space, there is the notion of a coherent module, see \cite[5.3]{EGA0}. 

\begin{defn}
Let $X$ be an adic space (locally of finite type over $\Q_p$) and let $E$ be a sheaf of $\Ocal_X^+$-modules on $X$.\\
\noindent (i) The $\Ocal_X^+$ module $E$ is called \emph{of finite type} or \emph{finitely generated}, if there exist an open covering $X=\bigcup_{i\in I} U_i$ and  for all $i\in I$ exact sequences
\[(\Ocal^+_{U_i})^{d_i}\longrightarrow E|_{U_i}\longrightarrow 0.\] 
\noindent (ii) The module is called \emph{coherent}, if it is of finite type and for any open subspace $U\subset X$ the kernel of any morphism $(\Ocal_U^+)^d\rightarrow E|_U$ is of finite type. 
\end{defn}
\begin{rem}
Let $X$ be a reduced adic space locally of finite type over $\Q_p$. Then locally on $X$ the sections $\Gamma(X,\Ocal_X)$ as well as $\Gamma(X,\Ocal_X^+)$ are noetherian rings. Hence the notion of a coherent $\Ocal_X^+$-module and the notion of an $\Ocal_X^+$-module of finite type coincide for these spaces. Especially, an $\Ocal_X^+$-module which is locally associated with a module of finite type is coherent. 
\end{rem}
\begin{rem}
The same definition of course also applies to the sheaves of period rings that we defined above. 
\end{rem}
Let  $X=\Spa(A,A^+)$ be an affinoid adic space. Then any finitely generated $A^+$-module $M$ defines a coherent sheaf of $\Ocal_X^+$-modules $E$ by the usual procedure  
\[\Gamma(\Spa(B,B^+),E)=M\otimes_{A^+}B^+\]
for an affinoid open subspace $\Spa(B,B^+)\subset X$. 

On the other hand it is not true that all coherent $\Ocal_X^+$-modules on an affinoid space arise in that way, as shown by the following example. The reason is that the cohomology $H^1(X,E)$ of coherent $\Ocal_X^+$-sheaves does not necessarily vanish on affinoid spaces.
\begin{expl}
Let $X=\Spa(\Q_p\langle T\rangle,\Z_p\langle T\rangle)$ be the closed unit disc. Let 
\begin{align*}
U_1&=\{x\in X\mid |x|\leq |p|\} \\ U_2&=\{x\in X\mid |p|\leq |x|\leq 1\}.
\end{align*}
Define the $\Ocal_X^+$-sheaf $E_1\subset \Ocal_X$ by glueing $\Ocal_{U_1}^+$ and $p^{-1}T\Ocal_{U_2}^+$ over $U_1\cap U_2$ and $E_2\subset \Ocal_X$ by glueing $\Ocal_{U_1}^+$ and $pT^{-1}\Ocal_{U_2}$. 
Then $E_1$ and $E_2$ are coherent $\Ocal_X^+$-modules. We have 
\begin{align*}
\Gamma(X,E_1)&=(1,p^{-1}T)\Gamma(X,\Ocal_X^+),\\
\Gamma(X,E_2)&=p\Gamma(X,\Ocal_X^+).
\end{align*}
Especially $E_2$ is not generated by global sections.
If $\Xcal=\Ucal_1\cup\Ucal_2\cong \widehat{\Pbb}^1_{\Z_p}\cup\widehat{\Abb}^1_{\Z_p}$ is the canonical formal model of $X=U_1\cup U_2$, then $E_2$ is defined by the coherent $\Ocal_\Xcal$-sheaf which is trivial on the formal affine line and which is the twisting sheaf $\Ocal(1)$ on the formal projective line, while $E_1$ is defined by the dual of the twisting sheaf $\Ocal(-1)$ on the formal projective line.
\end{expl}

Let $X$ be an adic space of finite type over $\Q_p$ (especially $X$ is quasi-compact) and $E$ be a coherent $\Ocal_X^+$-module on $X$. As $E$ is not necessarily associated to an $A^+$-module on an affinoid open $\Spa(A,A^+)\subset X$, the sheaf $E$ does not necessarily have a model $\Ecal$ over any formal model $\Xcal$ of $X$: The sheaf $\Ucal\mapsto\Gamma(\Ucal^{\ad},E)$ does not define $E$ in the generic fiber in general.
However there is a covering $X=\bigcup U_i$ of $X$ by finitely many open affinoids such that $E|_{U_i}$ is the sheaf defined by the finitely generated $\Gamma(U_i,\Ocal_X^+)$-module $\Gamma(U_i,E)$.
Hence there is a formal model $\Xcal$ of $X$ such that $E$ is defined by a coherent $\Ocal_\Xcal$-modules $\Ecal$. Namely $\Xcal$ is a formal model on which one can realize the covering $X=\bigcup U_i$ as a covering by open formal subschemes.

Let $E$ be a coherent $\Ocal_X^+$-module on an adic space $X$ and let $x\in X$. Let $\mfrak_x\subset \Ocal_{X,x}$ denote the maximal ideal of function vanishing at $x$ and write $\mfrak_x^+=\mfrak_x\cap\Ocal_{X,x}^+$, i.e. $\Ocal_{X,x}^+/\mfrak_x^+=k(x)^+$ is the integral subring of $k(x)$. 
We write $E\otimes k(x)^+$ for the fiber of $E$ at $x$, that is for the quotient of the $\Ocal_{X,x}^+$-module 
\[E_x=\lim\limits_{\substack{\longrightarrow\\ U\ni x}} \Gamma(U,E)\]
by the ideal $\mfrak_x^+$. 

Let $\Xcal$ be a formal model of $X$ and $\Ecal$ be a coherent $\Ocal_\Xcal$-module defining $E$ in the generic fiber. Further let $\Spf k(x)^+\hookrightarrow \Xcal$ denote the morphism defining $x$ in the generic fiber.
Then $\Ecal\otimes k(x)^+=E\otimes k(x)^+.$ If we write $\bar\Xcal$ for the special fiber of $\Xcal$ and $\bar \Ecal$ for the restriction of $\Ecal$ to $\bar\Xcal$ and if $x_0\in\bar\Xcal$ denotes the specialization of $x$, then it follows that
\[\bar\Ecal\otimes k(x_0)=(\Ecal\otimes k(x)^+)\otimes_{k(x)^+}k(x_0)=(E\otimes k(x)^+)\otimes_{k(x)^+}k(x_0).\] 

\begin{defn}
Let $E$ be a vector bundle of rank $d$ on an adic space $X$, locally of finite type over $\Q_p$.  A \emph{lattice} in $E$ is a coherent $\Ocal_X^+$-submodule $E^+\subset E$ which is locally on $X$ free of rank $d$ and which generates $E$, i.e. the inclusion induces an isomorphism 
\[E^+\otimes_{\Ocal_X^+}\Ocal_X\cong E.\]
\end{defn}

\begin{lem}\label{latticeiftorsfree}
Let $X$ be an adic space of finite type over $\Q_p$ and let $E^+$ be a finitely generated $\Ocal_X^+$-submodule of $\Ocal_X^d$ which contains a lattice of $\Ocal_X^d$. Then $E^+$ is a lattice if and only if $E^+\otimes k(x)^+$ is $\varpi_x$-torsion free for all rigid points $x\in X$, where $\varpi_x\in k(x)^+$ is  a uniformizer.
\end{lem}
\begin{proof}
As $E^+$ is finitely generated, it is in fact coherent. Choose a formal model  $\Xcal$ and a coherent sheaf  $\Ecal^+$ on $\Xcal$ which is a model for $E^+$. We may assume that $\Ecal^+$ has no $p$-power torsion.\\
Let ${\rm sp}:X\rightarrow \bar \Xcal$ denote the specialization map to the reduced special fiber $\bar\Xcal$ of $\Xcal$. Further we write $\bar \Ecal^+$ for the restriction of $\Ecal^+$ to $\bar\Xcal$.
 Let $x_0\in \bar \Xcal$ be a closed point and let $x\in X$ be a rigid analytic point (i.e. $k(x)$ is a finite extension of $\Q_p$) with ${\rm sp}(x)=x_0$. Then we have 
\[\bar\Ecal^+\otimes k(x_0)=(E^+\otimes k(x)^+)\otimes_{k(x)^+}k(x_0).\]
On the other hand, as $E^+\otimes k(x)^+$ has no $\varpi_x$-torsion,  the $k(x)^+$-module $E^+\otimes k(x)^+$ is a submodule of $\Ocal_X^d\otimes k(x)=k(x)^d$. Further it is finitely generated and contains a basis of $k(x)^d$. Hence it is freely generated by $d$ elements. It follows that $\bar\Ecal^+\otimes k(x_0)$ has dimension $d$. As $\bar \Xcal$ is reduced and $\bar\Ecal^+$ is a coherent sheaf it follows that it is a vector bundle. Locally on $\Xcal$ we can lift $d$ generators of $\bar\Ecal^+$ to $\Ecal^+$. By Nakayamas lemma these lifts generate $\Ecal^+$ and hence they also generate the $\Ocal_X^+$-module $E^+$. On the other hand $\Ocal_X^d=E^+[1/p]$ is free on $d$ generators. Hence the generators of $E^+$ do not satisfy any relations.  
\end{proof}

\section{$(\phi,\Gamma)$-modules over the relative Robba ring}
In this section we define certain families of $\phi$-modules that will appear in the context of families of Galois representations later on. The main results of this section are already contained in \cite[6]{families}.

\begin{defn} Let $X$ be an adic space and $\Rscr\in\{\Acal_{X,K}, \Acal_X^\dagger\}$.\\
An \'etale $\phi$-module over $\Rscr$ is a coherent $\Rscr$-module $N$ together with an isomorphism \[\Phi:\phi^\ast N\longrightarrow N\]
\end{defn}
\begin{defn} Let $X\in \Ad^{\lft}_{\Q_p}$ and \[\Rscr\in \{\Bcal_{X,K}, \Bcal_X^\dagger,\Bcal_{X,\rig}^\dagger\}.\]
Write $\Rscr^+\subset \Rscr$ for the corresponding integral subring\footnote{The integral subring of $\Bcal_{X,\rig}^\dagger$ is $\Acal_X^\dagger$.}.\\
\noindent (i) A $\phi$-module over $\Rscr$ is a locally free $\Rscr$-module $N$ together with an isomorphism $\Phi:\phi^\ast N\rightarrow N$.\\ 
\noindent (ii) A $\phi$-module over $\Rscr$ is called \emph{\'etale} if it is locally on $X$ induced from a locally free \'etale $\phi$-module over $\Rscr^+$. 
\end{defn}
Recall that $K_\infty$ is a fixed cyclotomic extension of $K$ and $\Gamma={\rm Gal}(K_\infty/K)$ denotes the Galois group of $K_\infty$ over $K$.
\begin{defn}
Let $X\in\Ad_{\Q_p}^{\lft}$ and $\Rscr$ be any of the sheaves of rings defined above. \\
\noindent (i) A $(\phi,\Gamma)$-module over $\Rscr$ is a $\phi$-module over $\Rscr$ together with a continuous semi-linear action of $\Gamma$ commuting with the semi-linear endomorphism $\Phi$.\\
\noindent (ii) A $(\phi,\Gamma)$-module over $\Rscr$ is called \'etale if its underlying $\phi$-module is \'etale.
\end{defn}
\subsection{The \'etale locus}
If $X$ is an adic space (locally of finite type over $\Q_p$) and $x\in X$ is any point, we will write $\iota_x:x\rightarrow X$ for the inclusion of $x$. 
If $\Rscr$ is any of the sheaf of topological rings above and if $\Ncal$ is a sheaf of $\Rscr_X$-modules on $X$, we write \[\iota_x^\ast\Ncal=\iota_x^{-1}\Ncal\otimes _{\Rscr_X}\Rscr_x\]
for the pullback of $\Ncal$ to the point $x$.
The following result is a generalization of \cite[Theorem 7.4]{KedlayaLiu} to the category of adic spaces. 
\begin{theo}
Let $X$ be a reduced adic space locally of finite type over $\Q_p$ and $\Ncal$ be a family of $(\phi,\Gamma)$-modules over $\Bcal_{X,\rig}^\dagger$. \\
\noindent {\rm (i)} The set 
\[X^{\rm int}=\{x\in X\mid \iota_x^\ast\Ncal\ \text{is \'etale}\}\subset X\]
is open.\\
\noindent {\rm (ii)} There exists a covering $X^{\rm int}=\bigcup U_i$ and locally free \'etale $\Acal_{U_i}^\dagger$-modules $N_i\subset \Ncal|_{U_i}$ which are stable under $\Phi$ such that 
\[N_i\otimes_{\Acal_{U_i}^\dagger}\Bcal_{U_i,\rig}^\dagger=\Ncal|_{U_i},\] i.e. $\Ncal|_{X^{\rm int}}$ is \'etale. 
\end{theo}
\begin{proof}
This is \cite[Corollary 6.11]{families}. In loc.~cit.~we use a different Frobenius $\phi$. However the proof works verbatim in the case considered here. 
\end{proof}
\begin{theo}
Let $f:X\rightarrow Y$ be a morphism of reduced adic spaces locally of finite type over $\Q_p$. Let $N_Y$ be a family of $(\phi,\Gamma)$-modules over $\Bcal_{Y,\rig}^\dagger$ and write $N_X$ for its pullback over $\Bcal_{X,\rig}^\dagger$. Then $f^{-1}(Y^{\rm int})=X^{\rm int}$
\end{theo}
\begin{proof}
This is \cite[Proposition 6.14]{families}. Again the same proof applies with the Frobenius considered here.
\end{proof}

\subsection{Existence of \'etale submodules}
For later applications to Galois representations the existence of an \'etale lattice locally on $X$ will not be sufficient. We cannot hope that the \'etale lattices glue together to a global \'etale lattice on the space $X$. However we have a replacement which will be sufficient for applications.

\noindent {\bf Convention:}
Let $X$ be areduced adic space locally of finite type over $\Q_p$ and let $(\Ncal,\Phi)$ be an \'etale $\phi$-module over $\Bcal_{X,\rig}^\dagger$ and $(\hat\Ncal,\hat\Phi)$ be an (\'etale) $\phi$-module over $\Bcal_{X,K}$. We say that $(\hat\Ncal,\hat\Phi)$ \emph{is induced from} $(\Ncal,\Phi)$ if there exists a covering $X=\bigcup U_i$ and \'etale $\Acal_{U_i}^\dagger$ lattices $N_i\subset \Ncal|_{U_i}$ such that 
\[(\hat\Ncal,\hat\Phi)|_{U_i}=\big((N_i,\Phi)^\wedge\big)[\tfrac{1}{p}].\]
Note the every \'etale $\phi$-module over $\Bcal_{X,\rig}^\dagger$ gives rise to a unique $\phi$-module over $\Bcal_{X,K}$, as an \'etale $\Acal_X^\dagger$-lattice is unique up to $p$-isogeny.
\begin{prop}\label{exetalephimod}
Let $X$ be a reduced adic space of finite type and $(\hat\Ncal,\hat\Phi)$ be a $\phi$-module over $\Bcal_{X,K}$ which is induced from an \'etale $\phi$-module $(\Ncal,\Phi)$ over $\Bcal_{X,\rig}^\dagger$. Then there exists an \'etale $\phi$-submodule $\hat N\subset \hat\Ncal$ over $\Acal_{X,K}$ such that the inclusion induces an isomorphism after inverting $p$. 
\end{prop}
\begin{prop}\label{exofcoherentmodules}
Let $X$ be an reduced adic space of finite type over $\Q_p$. Let $\Ncal$ be a locally free $\Bcal_{X,\rig}^{\dagger,r}$-module, then there exists a coherent $\Acal_X^{\dagger,r}$-submodule $N\subset \Ncal$ which {\rm (}locally on $X${\rm )} contains a basis of $\Ncal$. 
\end{prop}
\begin{proof}
Let $X=\bigcup_{i=1}^m U_i$ be a finite covering such that $\Ncal|_{U_i}$ is free and write $V_i=\bigcup_{j=1}^iU_j$. As obviously there exists an $\Acal_{U_i}^{\dagger,r}$-lattice in $\Ncal|_{U_i}$ it is enough to show that there is an extension of such a module from $V_i$ to $V_{i+1}$. This reduces the claim to the following lemma.
\end{proof}

\begin{lem}
Let $X=\Spa(A,A^+)$ be a reduced affinoid adic space and $U\subset X$ an quasi-compact open subset. Let $N_U$ be a finitely generated $\Acal_U^{\dagger,r}$-submodule of $(\Bcal_{U,\rig}^{\dagger,r})^d$ which contains  a basis. Then there exists a coherent $\Acal_X^{\dagger,r}$-module $N_X\subset (\Bcal_{X,\rig}^{\dagger,r})^d$ such that $N_X$ contains a basis and  such that $N_X|_U=N_U$. \\
\end{lem}
\begin{proof}
Let $N', N''\subset (\Bcal_{X,\rig}^{\dagger, r})^d$ be $\Acal_X^{\dagger, r}$-lattices such that $N''|_U\subset N_U\subset N'|_U$. After localizing we may assume that $N'$ is free. Denote by $j:U\hookrightarrow X$ the open embedding of $U$. We define $\tilde N_X$ by
\[\tilde N_X=\ker (N'\longrightarrow (j\times\id)_\ast (N'_U/N_U)).\] 
This is easily seen to be a coherent sheaf on $X\times \Ubb_{\geq p^{-1/r}}$.  We define $N_X$ by 
\[X\supset V\longmapsto \Gamma(V\times \Ubb_{\geq p^{-1/r}},\tilde N_X).\]
It is obvious that $N''\subset N_X\subset N'$ and hence $N_X$ contains a basis of  $(\Bcal_{X,\rig}^{\dagger,r})^d$.
It remains to check that this sheaf is coherent. 
Let $U=\bigcup U_i$ be a finite covering by open affinoids such that $N_U$ is associated to a finitely generated $\Gamma(U_i,\Acal_X^{\dagger,r})$-module. Choose a covering $X=\bigcup V_j$ by open affinoids such that $V_j\cap U\subset U_{i_j}$ for some index $i_j$. Then $N_X$ is associated to the $\Gamma(V_j,\Acal_X^{\dagger,r})$-module
\[\ker\big(\Gamma(V_j,N')\longrightarrow \Gamma(U_{i_j},N'_U/N_U)\otimes_{\Gamma(U_{i_j},\Acal_X^{\dagger,r})}\Gamma(V_j\cap U,\Acal_X^{\dagger,r})\big).\]
Especially $N_X$ is quasi-coherent. Finally $N_X$ is coherent as the sections of $\Acal_X^{\dagger, r}$ are locally on $X$ noetherian rings, and  $N_X\subset N'$.
\end{proof}
\begin{proof}[Proof of Proposition $\ref{exetalephimod}$]
As $X $ is quasi-compact, we can choose a locally free model $(\Ncal_r,\Phi_r)$ of $(\Ncal,\Phi)$ over $\Bcal_{X,\rig}^{\dagger,r }$ for some $r\gg 0$. After enlarging $r$ if necessary, we can assume that there exists a finite covering $X=\bigcup U_i$ and \'etale lattices $M_i\subset \Ncal_r|_{U_i}$.  By Proposition $\ref{exofcoherentmodules}$, there exist coherent $\Acal_X^{\dagger, r}$-modules $\tilde M_1\subset N_0 \subset \tilde M_2\subset \Ncal$
such that 
\[\tilde M_1|_{U_i}\subset N_0|_{U_i}\subset M_i \subset \tilde M_2|_{U_i}.\]
Let $\Ncal_{r_i}$ denote the restriction of $\Ncal_r$ to $\Bcal_{X,\rig}^{\dagger r_i}$, where we write $r_i=p^ir$. Then we inductively define coherent $\Acal_X^{\dagger, r_i}$-modules $N_i\subset \Ncal_{r_i}$ by setting
\[N_{i+1}=N_i\otimes_{\Acal_X^{\dagger, r_i}} \Acal_X^{\dagger, r_{i+1}}+\Phi(\phi^\ast N_i).\]
By assumption, we always have
\[\tilde M_1|_{U_i}\otimes_{\Acal_{U_i}^{\dagger, r}}\Acal_{U_i}^{\dagger, r_j}\subset N_j|_{U_i}\subset M_i\otimes_{\Acal_{U_i}^{\dagger, r}}\Acal_{U_i}^{\dagger, r_j}\subset \tilde M_2|_{U_i}\otimes_{\Acal_{U_i}^{\dagger, r}}\Acal_{U_i}^{\dagger, r_j}.\]
Viewing $N_i$ as coherent sheaves on $X\times \Ubb_{\geq p^{-1/r_i}}$, we now define an $\Acal_X^\dagger$-submodule $N\subset \Ncal$, by setting
\[N=\big(\lim\limits_{\substack{{\longrightarrow}\\{i\in \mathbb{N}}}}\pr_{i,\ast}N_i\big)\otimes \Acal_X^\dagger,\]
where $\pr_i:X\times \Ubb_{\geq p^{-1/r_i}}\rightarrow X$ is the projection to the first factor.
By construction, this module satisfies 
\begin{align*}
\Phi(\phi^\ast N)&\subset N, \\
\tilde M_1\otimes_{\Acal_X^{\dagger, r}}\Acal_X^\dagger &\subset N\subset \tilde M_2\otimes_{\Acal_X^{\dagger, r}} \Acal_X^\dagger.
\end{align*}
We then take $\hat N$ to be the completion of $N$ with respect to the $p$-adic topology.
If the module we started with is associated to a finitely generated module over an affinoid open $U\subset X$, then the construction implies that $\hat N$ is also associated with a $\Gamma(U,\Acal_{X,K})$-module which is contained in the finitely genrated module $\Gamma(U,\tilde M_2\otimes \Acal_{X,K})$ and hence has to be finitely generated itself. It follows that $\hat N$ is coherent, as claimed.
Further the construction implies that  $\hat N$ is a submodule of $\hat \Ncal$ and 
\begin{align*}
\hat\Phi(\phi^\ast \hat N)&\subset \hat N,\\
\hat N\otimes_{\Acal_{X,K}}\Bcal_{X,K}&=\hat\Ncal.
\end{align*}
We need to show that $\hat\Phi(\phi^\ast \hat N)=\hat N$. To do so, we can work locally on $X$ and hence assume that $\hat N$ is contained in an \'etale $\Acal_{X,K}$-lattice $\hat M\subset \hat\Ncal$.
Further is is enough to assume that $X=\Spa(A,A^+)$ is affinoid and we need to show that for all maximal ideals $\mfrak\subset A^+$ we have
\begin{equation}\label{Phimodm}
\hat\Phi(\phi^\ast\hat N)\otimes k_\mfrak=\hat N\otimes k_\mfrak,
\end{equation}
where $k_\mfrak=A^+/\mfrak$ denotes the residue field at $\mfrak$.\\
For a rigid analytic point $x\in X$, the fiber $\hat N\otimes k(x)^+$ is a finitely generated module over the ring $\Acal_{X,K}\otimes k(x)^+$ which is (a product of) complete discrete valuation rings.
Write \[\big(\hat N\otimes k(x)^+\big)^{\rm tors-free}\subset \hat N\otimes k(x)^+\]
for the submodule which is $\varpi_x$-torsion free. This submodule has to be free and 
\[\big(\hat N\otimes k(x)^+\big)^{\rm tors-free}[\tfrac{1}{p}]=(\hat M\otimes k(x)^+)[\tfrac{1}{p}]=\hat{\Ncal}.\]
It follows from Lemma $\ref{phistabetale}$ below that $\big(\hat N\otimes k(x)^+\big)^{\rm tors-free}$ is an \'etale $\phi$-module, i.e. $\hat\Phi$ is surjective.
We reduce the inclusion $\hat\Phi:\phi^\ast \hat N\hookrightarrow \hat N$ modulo $\mfrak$ and obtain a morphism
\[\bar\Phi:\phi^\ast \bar N\longrightarrow \bar N.\]
Assume there exists $0\neq \bar f\in\ker\bar\Phi$. Then there exists a rigid point $x\in X$ such that $\mfrak_x^+\subset \mfrak$ and a lift $f$ of $\bar f$ in the torsion-free part of $\hat N\otimes k(x)^+$ such that $f\notin \varpi_x(\hat N\otimes k(x)^+)$, as $\hat N$ is $\Z_p$-flat. It follows that $\hat \Phi(f)\in\varpi_X(\hat N\otimes k(x)^+)$ and hence $f\in\varpi_x(\hat N\otimes k(x)^+)$, as $\hat \Phi$ induces an isomorphism 
\[\phi^\ast\big(\hat N\otimes k(x)^+\big)^{\rm tors-free}\longrightarrow \big(\hat N\otimes k(x)^+\big)^{\rm tors-free}.\]
We have shown that $\ker \bar\Phi=0$ and hence $\hat\Phi$ is injective modulo all maximal ideals. 
By comparing dimension, we find that $(\ref{Phimodm})$ holds true for all maximal ideals $\mfrak\subset A^+$, and hence $\hat \Phi$ induces an isomorphism on $\hat N$.
\end{proof}
\begin{lem}\label{phistabetale}
Let $F$ be a finite extension of $\Q_p$ and $(\hat N,\hat\Phi)$ be a free \'etale $\phi$-module over $\Acal_{F,K}$. Let $\hat N_1\subset \hat N$ be a finitely generated submodule such that $\hat N_1[1/p]=\hat N[1/p]$ and $\hat\Phi(\phi^\ast\hat N_1)\subset \hat N_1$. Then $(\hat N_1,\hat \Phi)$ is an \'etale $\phi$-module, i.e. 
\[\hat\Phi(\phi^\ast \hat N_1)=\hat N_1.\]
\end{lem}
\begin{proof}
As $\Acal_{F,K}$ is (a product of) discrete valuation rings, it is clear that $\hat N_1$ is free on $d$ generators, where $d$ is the $\Acal_{F,K}$-rank of $\hat N$. Let $b_1,\dots,b_d$ be a basis of $\hat N$ and $e_1,\dots,e_d$ be a basis of $\hat N_1$. Let $A$ denote the change of basis matrix from $\underline{b}$ to $\underline{e}$ and denote by ${\rm Mat}_{\underline{b}}(\hat \Phi)$ resp. ${\rm Mat}_{\underline{e}}(\hat \Phi)$ the matrix of $\hat\Phi$ in the basis $\underline{b}$ resp. $\underline{e}$ of $\hat N[1/p]=\hat N_1[1/p]$. Then our assumptions imply that 
\[{\rm Mat}_{\underline{e}}(\hat \Phi)\in{\rm Mat}_{d\times d}(\Acal_{F,K}).\]
On the other hand 
\[{\rm Mat}_{\underline{e}}(\hat \Phi)=A^{-1} {\rm Mat}_{\underline{b}}(\hat \Phi)\phi(A)\]
and hence $\det {\rm Mat}_{\underline{e}}(\hat \Phi)\in \Acal_{F,K}^\times$, as $\hat N$ is \'etale and \[{\rm val}_p(\det A)={\rm val}_p(\det \phi(A)).\]
\end{proof}

\section{Families of $p$-adic Galois representations}
In this section we study the relation between Galois representations and $(\phi,\Gamma)$-modules in families. This problem was first considered by Berger and Colmez in \cite{BergerColmez}, where they define a functor from families of $G_K$-representations to families of overconvergent $(\phi,\Gamma)$-modules. 

\begin{defn} Let $G$ a topological group and $X$ an adic space locally of finite type over $\Q_p$. A family of $G$-representations over $X$ is a vector bundle $\Vcal$ over $X$ endowed with a continuous $G$-action.
\end{defn}
We write ${\rm Rep}_XG$ for the category of families of $G$-representations over $X$.  Recall that  we write $G_K={\rm Gal}(\bar K/K)$ for the absolute Galois group of a fixed local field $K$. In this case Berger and Colmez define  the functor
 \[{\bf D}^\dagger:{\rm Rep}_X G_K\longrightarrow \{\text{\'etale}\ (\phi,\Gamma)\text{-modules over }\ \Bcal_X^\dagger\},\]
 which maps a family $\Vcal$ of $G_K$-representations on $X$ to the \'etale $(\phi,\Gamma)$-module
 \[{\bf D}^\dagger(\Vcal)=\big(\Vcal\otimes_{\Ocal_X}\tilde\Bcal_X^\dagger\big)^{H_K}.\]
 More precisely they construct this functor if $X$ is a reduced affinoid adic space of finite type.  As the functor ${\bf D}^\dagger$ is fully faithful in this case and maps $\Vcal$ to a free $\Bcal_X^\dagger$-module it follows that we can consider ${\bf D}^\dagger$ on the full category ${\rm Rep}_XG_K$, whenever $X$ is reduced. In the following we will always assume that $X$ is reduced.
 
 We will consider the variant 
 \[{\bf D}^\dagger_{\rig}:\Vcal\longmapsto \big(\Vcal\otimes_{\Ocal_X}\tilde\Bcal_{X,\rig}^\dagger\big)^{H_K} ={\bf D}^\dagger(\Vcal)\otimes_{\Bcal_X^\dagger}\Bcal_{X,\rig}^\dagger.\] 
Note that for an adic space $X$ of finite type over $\Q_p$, the $(\phi,\Gamma)$-module ${\bf D}^\dagger(\Vcal)$ is always defined over some $\Bcal_X^{\dagger,s}\subset \Bcal_X^\dagger$, for $s\gg 0$. Especially an \'etale lattice can be defined over $\Acal_X^{\dagger,s}$ for $s\gg 0$.
\subsection{The admissible locus}
It is known that the functors ${\bf D}^\dagger$ and ${\bf D}_{\rig}^\dagger$ are not essentially surjective. In \cite{KedlayaLiu}, Kedlaya and Liu construct a \emph{local inverse} to this functor. More precisely, they show that if $\Ncal$ is a family of $(\phi,\Gamma)$-modules over $\Bcal_{X,\rig}^\dagger$, then every rigid analytic point at which $\Ncal$ is \'etale has an affinoid neighborhood on which the family $\Ncal$ is the image of a family of $G_K$-representations.
However we need to extend this result to the setup of adic spaces in order to define a natural \emph{subspace} over which such a family $\Ncal$ is induced by a family of $G_K$-representations.  
 
 \begin{theo}\label{thmXadm}
 Let $X$ be a reduced adic space locally of finite type over $\Q_p$ and let $\Ncal$ be a family of $(\phi,\Gamma)$-modules of rank $d$ over $\Bcal_{X,\rig}^\dagger$. \\
 \noindent {\rm (i)} The subset 
 \[X^{\rm adm}=\big\{x\in X\mid \dim_{k(x)}\big((\Ncal\otimes_{\Bcal_{X,\rig}^\dagger}\tilde\Bcal_{X,\rig}^\dagger)\otimes k(x)\big)^{\Phi=\id} =d\big\}\]
 is open.\\
 \noindent {\rm (ii)} There exists a family of $G_K$-representations $\mathcal{V}$ on $X^{\rm adm}$ such that there is a canonical and functorial isomorphism 
 \[{\bf D}^\dagger_{\rig}(\mathcal{V})\cong \Ncal|_{X^{\rm adm}}.\]
 \noindent {\rm (iii)} Let $\Vcal$ be a family $G_K$-representations on $X$ such that ${\bf D}_{\rig}^\dagger(\Vcal)=\Ncal$. Then $X^{\rm adm}=X$.
 \end{theo}
 Let $A$ be a complete topological $\Q_p$-algebra and let $A^+\subset A$ be a ring of integral elements. Assume that the completed tensor products $A^+\widehat{\otimes}\tilde\bfA^\dagger$ and $A\widehat{\otimes}\tilde\bfB^\dagger_{\rig}$ are defined\footnote{The examples we consider here, are $\Gamma(X,\Ocal_X)$ for an affinoid adic space of finite type and the completions of $k(x)$ for a point $x\in X$. In the latter case the completed tensor product is the completion of the fiber of $\tilde\Acal^\dagger$ resp. $\tilde\Bcal^\dagger_{\rig}$ at the point $x$.}. In this case the following approximation Lemma of Kedlaya and Liu applies.
 \begin{lem}\label{approxLemma}
 Let $\tilde\Ncal$ be a free $(\phi,\Gamma)$-module over $A\widehat{\otimes}\tilde\bfB_{\rig}^\dagger$ such that there exists a basis on which $\Phi$ acts via $\id+B$ with
 \[B\in p{\rm Mat}_{d\times d}(A^+\widehat{\otimes}\tilde\bfA^\dagger).\]
 Then $\tilde\Ncal^{\Phi=\id}$ is free of rank $d$ as an $A$-module.
 \end{lem}
 \begin{proof}This is \cite[Theorem 5.2]{KedlayaLiu}. \end{proof}
 \begin{cor}\label{admstbundercompletion}
 Let $X$ be an adic space locally of finite type over $\Q_p$ and $\tilde\Ncal$ be a family of $(\phi,\Gamma)$-modules over $\tilde\Bcal_{X,\rig}^\dagger$. Let $x$ in $X$, then 
 \[\dim_{\widehat{k(x)}}(\iota_x^{\ast}\tilde\Ncal)^{\rm \Phi=\id}=d\ \Longleftrightarrow \dim_{k(x)}\big((\Ncal\otimes_{\Bcal_{X,\rig}^\dagger}\tilde\Bcal_{X,\rig}^\dagger)\otimes k(x)\big)^{\Phi=\id} =d.\]
 \end{cor}
 \begin{proof}
 The proof is the same as the proof of \cite[Proposition 8.20 (i)]{families}.
\end{proof}
\begin{proof}[Proof of Theorem $\ref{thmXadm}$]
Let $x\in X^{\rm adm}$ and denote by $Z$ the Zariski-closure of $x$, that is, the subspace defined by the ideal of all functions vanishing at $x$. This is an adic space locally of finite type. Let $U\subset Z$ be a affinoid neighborhood of $x$ in $Z$ such that a basis of the $\Phi$-invariants extends to $U$. It follows from Lemma $\ref{approxLemma}$ that 
\[\Vcal_U=\big(\Ncal|_Z\otimes_{\Bcal_{Z,\rig}^\dagger}\tilde\Bcal_{U,\rig}^\dagger\big)^{\Phi=\id}\]
is free of rank $d$ over $\Ocal_U$. On this sheaf we have the diagonal $G_K$-action given by the natural action on $\tilde\Bcal_{U,\rig}^\dagger$ and the $\Gamma$-action on $\Ncal$. It is a direct consequence of the construction that 
\[{\bf D}_{\rm rig}^\dagger(\Vcal_U)=\Ncal|_U.\]
Especially we have shown that $X^{\rm adm}\subset X^{\rm int}$.

Now let $x\in X^{\rm adm}$ and let $U$ denote a neighborhood of $x$ to which we can lift a basis of $\Phi$-invariants. As $\Ncal$ is known to be \'etale, we can shrink $U$ such that we are in the situation of Lemma $\ref{approxLemma}$. 

It follows that $X^{\rm adm}$ is open and that 
\[\big(\Ncal\otimes_{\Bcal_{X,\rm rig}^\dagger}\tilde\Bcal_{X,\rm rig}^\dagger\big)^{\Phi=\id}\]
gives a vector bundle $\Vcal$ on $X^{\rm adm}$. Again, we have the diagonal action of $G_K$. As above we find that 
\[{\bf D}^\dagger_{\rig}(\Vcal)=\Ncal|_{X^{\rm adm}\cap X^{\rm int}}.\]
Finally (iii) is obvious by the construction of \cite{BergerColmez}.
 \end{proof}
 \begin{theo}
 Let $f:X\rightarrow Y$ be a morphism of adic spaces locally of finite type over $\Q_p$ with $Y$ reduced. Further let $\Ncal_Y$ be a family of $(\phi,\Gamma)$-modules over $\Bcal_{Y,\rig}^\dagger$ and write $\Ncal_X$ for the pullback of $\Ncal_Y$ to $X$. Then $f^{-1}(Y^{\rm adm})=X^{\rm adm}$ and $f^\ast \mathcal{V}_Y=\Vcal_X$ on $X^{\rm adm}$.
 \end{theo}
 \begin{proof}
 Using the discussion above, the proof is the same as the proof of \cite[Proposition 8.22]{families}.
 \end{proof}
 \begin{prop}\label{Xadmpartproper}
Let $X$ be a reduced adic space locally of finite type over $\Q_p$ and let $\Ncal$ be a family of $(\phi,\Gamma)$-modules over $\Bcal_{X,\rig}^\dagger$. Then the inclusion
\[f:X^{\rm adm}\longrightarrow X\]
is open and partially proper.
\end{prop}
\begin{proof}
We have already shown that $f$ is open. Especially it is quasi-separated and hence we may apply the valuative criterion for partial properness, see \cite[1.3]{Huber}. 
Let $(x,A)$ be a valuation ring of $X$ with $x\in X^{\rm adm}$ and let $y\in X$ be a center of $(A,x)$. We need to show that $y\in X^{\rm adm}.$ As $y$ is a specialization of $x$, the inclusion $i:k(y)\hookrightarrow k(x)$ identifies $k(y)$ with a dense subfield of $k(x)$. Especially 
\[\tilde\Ncal_y:=\Ncal\otimes_{\Bcal_{k(y),\rig}^\dagger}\tilde\Bcal_{k(y),\rig}^\dagger\longrightarrow \Ncal\otimes_{\Bcal_{k(x),\rig}^\dagger}\tilde\Bcal_{k(x),\rig}^\dagger=:\tilde\Ncal_x\]
is dense. Let $e_1,\dots,e_d$ be a basis of $\tilde\Ncal_x$ on which $\Phi$ acts as the identity. We may approximate this basis by a basis of $\tilde\Ncal_y$. Thus we can choose a basis of $\tilde\Ncal_y$ on which $\Phi$ acts by $\id+A$ with 
\[A\in {\rm Mat}_{d\times d}(\tilde\Bcal_{k(y),\rig}^\dagger)\]
sufficently small. For example we can choose 
\[A\in p\, {\rm Mat}_{d\times d}(\tilde\Acal_{k(y)}^\dagger).\]
By Lemma $\ref{approxLemma}$ and Corollary $\ref{admstbundercompletion}$ it follows that $y\in X^{\rm adm}$.
\end{proof}

 \subsection{Existence of Galois representations}
 In this section we link deformations of Galois representations and deformations of \'etale $\phi$-modules. 

In the following $(R,\mfrak)$ will denote a complete local noetherian ring, topologically of finite type over $\Z_p$.  As above we have the notion of an \'etale $\phi$-module over \[R\widehat{\otimes}_{\Z_p}{\bf A}_K=\lim\nolimits_\leftarrow \big((R/\mfrak^n)\otimes_{\Z_p}\bfA_K\big).\]

A Galois representation with coefficients in $R$ (or a family of Galois representations on $\Spf R$) is a continuous representation 
\[G\longrightarrow \GL_d(R),\]
where $G$ is the absolute Galois group of some field $L$.
The relation between Galois representations and \'etale $\phi$-modules with coefficients in local rings was first considered by Dee, see \cite[2]{Dee}.

\begin{theo}\label{exgalreptheo}
Let $X$ be a reduced adic space of finite type over $\Q_p$ and and let $(\Ncal,\Phi)$ be a family of \'etale $\phi$-modules over $\Bcal_{X,\rig}^\dagger$. Let $x_0\in \bar \Xcal$ be a closed point in the special fiber of some formal model $\Xcal$of $X$ and let $Y\subset X$ denote the tube of $x_0$. Then $(\Ncal,\Phi)|_Y$ is associated to a family of $H_K$-representations on the open subspace $Y$.
\end{theo}
\begin{proof}
It follows from Proposition $\ref{exetalephimod}$ that there exist an \'etale $\phi$-module $\hat N$ over $\Acal_{X,K}$ such that $\hat N\subset \hat\Ncal$ as $\phi$-modules and such that $\hat N$ contains a basis of $\hat\Ncal$. Here $\hat\Ncal$ is the $\Bcal_{X,K}$-module induced from $\Ncal$. 
Choose an affine neighborhood $\Ucal=\Spf(A^+)$ of $x_0$ and write $U$ for its generic fiber. We write $\mfrak\subset A^+$ for the maximal ideal defining $x_0$ and write $R$ for the $\mfrak$-adic completion of $A^+$. Then $Y$ is the generic fiber of $\Spf R$. Write $\Nfrak=\Gamma(U,\hat N)$, then this is a finitely generated $\Gamma(U,\Acal_{X,K})$-module on which $\hat\Phi$ induces a semi-linear isomorphism. Especially $\hat\Nfrak=\Nfrak\widehat{\otimes}_{A^+} R$ is a finitely generated \'etale  $\phi$-module over $\Gamma(Y,\Acal_{X,K})=R\widehat{\otimes}_{\Z_p}{\bf A}_K$. Hence, by \cite{Dee},  there is a finitely generated $R$-module $E$ with continuous $H_K$ action associated with $\hat \Nfrak$. Then 
\[Y\supset V\mapsto E\otimes_R\Gamma(V,\Ocal_X)\]
defines the desired family of Galois representations\footnote{Note that we do not claim that locally on $Y$ the integral representation $E$ is associated with an \'etale lattice in $(\Ncal,\Phi)$. This is only true up to $p$-isogeny.} on $Y$ .
\end{proof}
\begin{cor}
Let $X$ be a reduced adic space locally of finite type over $\Q_p$ and $\Ncal$ be a family of \'etale $(\phi,\Gamma)$-modules on $X$. Let $x_0\in \bar \Xcal$ be a closed point in the special fiber of some formal model $\Xcal$ of $X$ and let $Y\subset X$ denote the tube of $x_0$. Then $\Ncal|_Y$ is associated to a family of $G_K$-representations on the open subspace $Y$.
\end{cor}
\begin{proof}
By the above theorem it follows that $Y=Y^{\rm adm}$. The claim follows from Theorem $\ref{thmXadm}$.
\end{proof}
\begin{conj}\label{Conjecture}
The claim of the theorem (and the corollary) also holds true if we replace $x_0$ by a (locally) closed subscheme of the special fiber over which there exists a Galois representation that is locally associated with the reduction of an \'etale lattice.
\end{conj}

 \subsection{Local constancy of the reduction modulo $p$}
Let $L$ be a finite extension of $\Q_p$ with ring of integers $\Ocal_L$, uniformizer $\varpi_L$ and residue field $k_L$. Let $V$ be a $d$-dimensional $L$-vector space with a continuous action of a compact group $G$.
We choose a $G$-stable $\Ocal_L$-lattice $\Lambda\subset V$ and write $\bar\Lambda=\Lambda/\varpi_L\Lambda$ for the reduction modulo the maximal ideal of $\Ocal_L$. Then $\bar\Lambda$ is a (continuous) representation of $G$ on a $d$-dimensional $k_L=\Ocal_L/\varpi_L\Ocal_L$-vector space. The representation $\bar\Lambda$ depends on the choice of a $G$-stable lattice $\Lambda\subset V$, however it is well known that its semisimplification $\bar\Lambda^{\rm ss}$(i.e. the direct sum of its Jordan-H\"older constituents) is independent of $\Lambda$ and hence only depends on the representation $V$. In the following we will write $\bar V$ for this representation and refer to it as the reduction modulo $\varpi_L$ of the representation $V$. 

The aim of this section is to show that the reduction modulo $\varpi_L$ is locally constant in a family\footnote{This seems to be a well known fact, ay least in the context of pseudo-characters. As we do not want to assume $p>d$ here, we give a different proof} of $p$-adic representations of $G$.  In the context of families of Galois representations this was shown by Berger for families of $2$-dimensional crystalline representations of  ${\rm Gal}(\bar\Q_p/\Q_p)$ in a weaker sense: Berger showed that every rigid analytic point has a neighborhood on which the reduction is constant, see \cite{Berger2}.

Let $X$ be an adic space locally of finite type over $\Q_p$ and $E$ a vector bundle on $X$ endowed with a continuous $G$-action. If $x\in X$, then we write 
\[\big (E\otimes \overline{k(x)}\big)=\big(\overline{E\otimes k(x)}\big)^{\rm ss}\]
for the semisimplification of the $G$-representation in the special fiber $\overline{k(x)}=k(x)^+/(\varpi_x)$ of $k(x)$.

We first claim that (up to semisimplification) there are no nontrivial families of representations of a finite group on varieties over $\Fbb_p$.
\begin{prop}\label{famofHrep} Let $H$ be a finite group. 
Let $X$ be a connected $\Fbb_p$-scheme of finite type and $\Ecal$ a vector bundle on $X$ endowed with an $H$-action. Then there is a semi-simple $H$-representation $E$ on a finite dimensional $\bar\Fbb_p$-vector space such that for all $x\in X$ there is an isomorphism of $H$-representations
\[\big((\Ecal\otimes k(x))\otimes_{k(x)}\bar\Fbb_p\big)^{\rm ss}\cong E.\]
\end{prop} 
\begin{proof}
For $h\in H$ consider the morphism 
\[ f_h:x\longmapsto {\rm charpoly}\big( h|\Ecal\otimes k(x)\big)\in k(x)^d,\]
mapping $x\in X$ to the coefficients of the characteristic polynomial of $h$ acting on $\Ecal\otimes k(x)$, where $d$ is the rank of $\Ecal$. This gives a morphism of schemes $X\rightarrow \Abb^d$.
As there are only finitely many isomorphism classes of semi-simple $H$-representations of fixed rank (there are only finitely many irreducible representations), this map has finite image and hence it has to be constant, as $X$ is connected.
It follows that for all $h\in H$ we have the equality
\[{\rm charpoly}(h|\Ecal\otimes k(x))={\rm charpoly}(h|\Ecal\otimes k(y))\]
for all $x,y\in X$. Then \cite[Theorem 30.16]{CurtisReiner} implies the claim.
\end{proof}
\begin{lem}
Let $X$ be an adic space locally of finite type and $E$ be a vector bundle on $X$ endowed with a continuous action of a compact group $G$. Then locally on $X$ there exists a $G$-stable $\Ocal_X^+$-lattice $E^+\subset E$.
\end{lem}
\begin{proof}
We may assume that $E\cong\Ocal_X^d$ is trivial and hence there is a lattice $E_1^+=(\Ocal_X^+)^d\subset E$. As $G$ is compact the entries of the matrices of $g\in G$ acting on the standard basis have a common bound. Hence the $\Ocal_X^+$-submodule $E^+\subset E$ which is generated by the $G$-translates of $E_1^+$ is contained in $p^{-N}E_1^+$ for some large integer $N$. Especially it is coherent. We need to show that it is a lattice, and hence by Lemma $\ref{latticeiftorsfree}$ we only need to show that the stalks are torsion free. But if $e_1,\dots,e_d$ are generators of $(k(x)^+)^d=E_1^+\otimes k(x)^+$, then the translates of $e_1,\dots,e_d$ under the action of $G$ generate the stalk $E^+\otimes k(x)^+$. It follows that the stalks are torsion free. 
\end{proof}
\begin{cor}\label{locconst}
Let $X$ be an adic space locally of finite type and let $E$ be a vector bundle on $X$ endowed with a continuous action of a compact group $G$.
Then the semi-simplification of the reduction $E\otimes \overline{k(x)}$ is locally constant. 
\end{cor}
\begin{proof}
As the statement is local on $X$, we may assume that $X$ is quasi-compact and admits a $G$-stable $\Ocal_X^+$-lattice $E^+\subset E$. Let $\Xcal$ be a $\Z_p$-flat formal model of $X$ such that there exists a model $\Ecal^+$ of $E^+$ on $\Xcal$. Then $\Ecal^+$ defines a continuous $G$-representations on the special fiber $\bar\Ecal^+$ which is a vector bundle on the special fiber $\bar \Xcal$ of $\Xcal$. As $G$ is compact and the representations is continuous, the representation on $\bar\Ecal^+$ has to factor over some finite quotient $H$ of $G$. Now the claim follows from Proposition $\ref{famofHrep}$, as $X$ is connected if and only if $\Xcal$ is connected.
\end{proof}

\section{A remark on slope filtrations}

In this section we give an explicit example of a family $(\Ncal,\Phi)$ of $\phi$-modules over the relative Robba ring which is  not \'etale, but \'etale at all rigid analytic points (and hence $(\Ncal,\Phi)$ is purely of slope zero).
For this section we use different notations. 
Let $K$ be a totally ramified quadratic extension of $\Q_p$. Fix a uniformizer $\pi\in\Ocal_K$ and a compatible system $\pi_n\in\bar K$ of $p^n$-th roots of $\pi$. Let us write $K_{\infty}=\bigcup K(\pi_n)$  and $G_{K_\infty}={\rm Gal}(\bar K/K_\infty)$ for this section. Further let $E(u)\in\Z_p[u]$ denote the minimal polynomial of $\pi$.
Finally we adapt the notation from \cite{families} and write 
\[\Bcal_X^R=\Bcal_{X,\rig}^\dagger\ \text{and}\ \Bcal_X^{[0,1)}=\pr_{X,\ast}\Ocal_{X\times\Ubb}.\]

We consider the following family $(D,\Phi,\Fcal^\bullet)$ of filtered $\phi$-modules on 
\[X=\Pbb^1_K\times \Pbb_K^1.\]
Let $D=\Ocal_X^2=\Ocal_X e_1\oplus \Ocal_X e_2$ and $\Phi={\rm diag}(\varpi_1,\varpi_2)$, where $\varpi_1$ and $\varpi_2$ are the zeros of $E(u)$. 
We consider a filtration $\Fcal^\bullet$ of $D_K=D\otimes_{\Q_p}K$ such that $\Fcal^0=D_K$ and $\Fcal^2=0$. Fix an isomorphism $D\otimes_{\Q_p}K\cong \Ocal_X^2\oplus\Ocal_X^2$ and let the  filtration step $\Fcal^1$ be the universal subspace on $X$. This is a family of filtered $\phi$-modules in the sense of \cite{families}.
One easily computes that 
\[X^{\rm wa}=X\backslash\{(0,0),(\infty,\infty)\},\]
where $X^{\rm wa}\subset X$ is the weakly admissible locus defined in \cite[4.2]{families}.
Generalizing a construction of Kisin \cite{Kisin} the family $(D,\Phi,\Fcal^\bullet)$ defines a family $(\Mcal,\Phi)$ consisting of a vector bundle on $X^{\rm wa}\times\Ubb$ and an injection $\Phi:\phi^\ast\Mcal\rightarrow \Mcal$ such that $E(u)\coker \Phi=0$ (see \cite[Theorem 5.4]{families}).

We define the family $(\Ncal,\Phi)$ over $\Bcal_{X^{\rm wa}}^R$ as 
\begin{equation}\label{defofN}
(\Ncal,\Phi)=(\Mcal,\Phi)\otimes_{\Bcal_{X^{\rm wa}}^{[0,1)}}\Bcal_{X^{\rm wa}}^R.
\end{equation}
This is obviously a family of $\phi$-modules over the Robba ring which is \'etale at all rigid analytic points.

\begin{prop}\label{notetale}
The family $(\Ncal,\Phi)$ over $\Bcal_{X^{\rm wa}}^R$ defined in $(\ref{defofN})$ is not \'etale.
\end{prop}
\begin{lem}
There exists a covering $X^{\rm wa}=\bigcup U_i$, where each $U_i$ is a closed disc or a closed annulus. Further there exists $x_i$ in the special fiber of the canonical formal model of $U_i$ such that $X^{\rm wa}=\bigcup V_i$, where $V_i\subset U_i$ is the tube of $x_i$.
\end{lem}
\begin{proof}
We can cover the weakly admissible set $X^{\rm wa}=X_1\cup X_2 \cup X_3 \cup X_4$, where 
\begin{align*}
X_1&= \big((\Pbb^1\backslash\{\infty\})\times (\Pbb^1\backslash\{\infty\})\big)\backslash\{(0,0)\}\cong \Abb^2\backslash\{0\},\\
X_2&=(\Pbb^1\backslash\{\infty\})\times (\Pbb^1\backslash\{0\})\cong \Abb^2,\\
X_3&=(\Pbb^1\backslash\{0\})\times (\Pbb^1\backslash\{\infty\})\cong \Abb^2,\\
X_4&=\big((\Pbb^1\backslash\{0\})\times(\Pbb^1\backslash\{0\})\big)\backslash\{(\infty,\infty)\}\cong \Abb^2\backslash\{0\}.
\end{align*}
Then obviously each of the $X_i$ can be covered by open subsets of the form  
\begin{align*}
 U&\cong \{x\in \Abb^2 \mid ||x||\leq s\}&&\text{for some}\ s>0,\\
 U&\cong \{x=(x_1,x_2)\in \Abb^2 \mid |x_i|\leq s_2,\ s_1\leq |x_1|\}&&\text{for some}\ 0<s_1\leq s_2,\\
 U&\cong \{x=(x_1,x_2)\in \Abb^2 \mid |x_i|\leq s_2,\ s_1\leq |x_2|\}&&\text{for some}\ 0<s_1\leq s_2,
\end{align*}
with $s,s_1,s_2\in p^\Q$. 

Choose a covering $U_i$ of $X^{\rm wa}$, where each $U_i$ is of the form described above, and let $\Ucal_i$ be the canonical formal model of $U_i$ with special fiber $\Abb^2$ resp. $\Abb^1\times (\Abb^1\cup \Abb^1)$. Here the two affine lines $\Abb^1\cup \Abb^1$ are glued together along the zero section. Let $V_i\subset U_i$ be the tube of the zero section. Then the $V_i$ also cover $X^{\rm wa}$.
\end{proof}

{\bf Claim:}\emph{
If the family $(\Ncal,\Phi)$ over $X^{\rm wa}$ defined by $(\ref{defofN})$ was \'etale, then there would exist a family of crystalline $G_K$-representations $\Ecal$ on $X^{\rm wa}$ such that }
\[D_{\rm cris}(\Ecal)=(D,\Phi,\Fcal^\bullet).\]

\begin{proof}[Proof of claim] Assume that the family $(\Ncal,\Phi)$ is \'etale. By Theorem $\ref{exgalreptheo}$\footnote{Strictly speaking Theorem $\ref{exgalreptheo}$ does not apply here, as we use a slightly different kind of Frobenius in \cite{families}. However the proof of the theorem works verbatim with the Frobenius used in loc.~cit.} there exists a family of $G_{K_\infty}$-representations on each of the $V_i$ which gives rise to the restriction of $\Mcal\otimes \Bcal_X^R$ to $V_i$.
Hence, by \cite[Theorem 8.25]{families}, there exists a family of crystalline $G_K$-representations on each of the $V_i$ giving rise to the restriction of our family of filtered isocrystals to $V_i$.
By the construction in \cite{families} this family is naturally contained in $D\otimes_{\Ocal_{V_i}} (\Ocal_{V_i}\widehat{\otimes}B_{\rm cris})$ and in fact identified with 
\[\Fil^0\big(D\otimes_{\Ocal_{V_i}} (\Ocal_{V_i}\widehat{\otimes}B_{\rm cris})\big)^{\Phi=\id}.\]
Further the space $X^{\rm wa}$ can be covered by the $V_i$. It follows that we can glue these families of crystalline representations to a family $\Ecal$ of crystalline $G_K$-representations on $X^{\rm wa}$ such that $D_{\rm cris}(\Ecal)=(D,\Phi,\Fcal^\bullet)$. 
\end{proof}
\begin{proof}[Proof of Proposition $\ref{notetale}$]
Assume by way of contradiction that the family $(\Ncal,\Phi)$ is \'etale. By the above claim there exists a family of $G_K$-representations $\Ecal$ on $X^{\rm wa}$ such that $V_{\rm cris}(D\otimes k(x))=\Ecal\otimes k(x)$ for all $x\in X$.

Now the space $X^{\rm wa}$ contains $K$-valued points $x_1,x_2$ and $x_3$ such that 
\begin{align*}
(\Mcal,\Phi)\otimes k(x_1)&\cong\left(\Ocal_{\Ubb_K}^2,\begin{pmatrix}0 & -E(u)\\1 & \varpi_1+\varpi_2\end{pmatrix}\right)\\ 
(\Mcal,\Phi)\otimes k(x_2)&\cong\left(\Ocal_{\Ubb_K}^2,\begin{pmatrix} 0 & -(u-\varpi_1)\\ (u-\varpi_2) & \varpi_1+\varpi_2\end{pmatrix}\right) \\
(\Mcal,\Phi)\otimes k(x_3)&\cong\left(\Ocal_{\Ubb_K}^2,\begin{pmatrix}-(u-\varpi_1) & 0 \\ 0 & -(u-\varpi_2)\end{pmatrix}\right).
\end{align*}
The semi-simplifications of the reduction modulo $\pi$ of these $\phi$-modules are 
\begin{align*} (\Mfrak,\Phi)\otimes \overline{ k(x_1)}&\cong\left(\Fbb_p[[u]]^2,\begin{pmatrix}0&-u^2\\ 1 & 0\end{pmatrix}\right),\\
 (\Mfrak,\Phi)\otimes \overline{ k(x_2)}&\cong\left(\Fbb_p[[u]]^2,\begin{pmatrix}0&-u\\ u &  0\end{pmatrix}\right),\\
 (\Mfrak,\Phi)\otimes \overline{ k(x_3)}&\cong\left(\Fbb_p[[u]]^2,\begin{pmatrix}-u&0\\ 0 & -u\end{pmatrix}\right).
 \end{align*}
Using Caruso's classification \cite[Corollary 8]{Caruso}  of those $\phi$-modules we find that they are all non-isomorphic.
After inverting $u$ these $\phi$-modules correspond (up to twist) under Fontaine's equivalence of categories to the restriction to $G_{K_\infty}$ of the reduction modulo $\pi$ of the constructed Galois representations $\Ecal\otimes k(x_i)$. By \cite[Theorem 3.4.3]{Breuil} this restriction is fully faithful and hence we find that 
\[\overline{\Ecal\otimes k(x_i)}\not\cong \overline{\Ecal\otimes k(x_j)}\]
as $G_K$-representations for $i\neq j$.

However, by Proposition $\ref{locconst}$, we know that the reduction modulo $p$ of the $G_K$-representation on the fibers of the family $\Ecal$ has to be constant. This is a contradiction, and hence the family $(\Ncal,\Phi)$ is not \'etale.
\end{proof}

\begin{rem}
This situation seems to be typical for the weakly admissible locus in the fibers over some fixed Frobenius $\Phi$. If we fix the Frobenius (or even its semi-simplification) then the main result of \cite{adquot} shows that the weakly admissible locus $X^{\rm wa}$ is a Zariski open subset of some flag variety $X$. It should be possible to cover $X^{\rm wa}$ by open subspaces $U_i$ such that there exists formal models $\Ucal_i$ and (locally) closed subschemes $Z_i\subset \bar\Ucal_i$ such that the tubes $V_i\subset U_i$ of $Z_i$ also cover $X^{\rm wa}$. If Conjecture $\ref{Conjecture}$ holds true, then the same argument as in this section shows that the weakly admissible family  is not \'etale unless all residual representations are isomorphic.
\end{rem}

This result also clarifies the relation between certain subspaces of a stack of filtered $\phi$-modules considered in \cite{families}. 
In loc. cit. we construct three open substacks of the stack $\Dfrak$ of filtered $\phi$-modules
\[\Dfrak^{\rm adm}\subset \Dfrak^{\rm int}\subset \Dfrak^{\rm wa}\subset \Dfrak.\]
The stack $\Dfrak^{\rm wa}$ parametrizes those families of  filtered $\phi$-modules which are weakly admissible, while the stack $\Dfrak^{\rm int}$ parametrizes those filtered $\phi$-modules such that the associated family of $\phi$-modules on the open unit disc admits an \'etale lattice after restriction to the (relative) Robba ring.
Finally $\Dfrak^{\rm adm}$ is the maximal (open) subspace over which there exists a family of crystalline Galois representations giving rise to the restriction of the universal family of filtered $\phi$-modules.
In \cite {families}, the stacks $\Dfrak^{\rm int}$ and $\Dfrak^{\rm adm}$ are only constructed in the case of Hodge-Tate weights in $\{0,1\}$.

As already stressed in \cite{families}, the inclusion $\Dfrak^{\rm adm}\subset \Dfrak^{\rm int}$ is strict, as is already shown by the family of unramified characters for example. The above example shows that the inclusion $\Dfrak^{\rm int}\subset \Dfrak^{\rm wa}$ is strict as well.

\end{document}